\documentclass{amsart}
\usepackage{a4wide}
\usepackage{tikz-cd}
\usepackage{enumerate}
\usepackage[utf8]{inputenc}
\usepackage{color}
\usepackage[linkbordercolor=red, citebordercolor=green, pdfborderstyle={/S/U/W 1}]{hyperref}

\hypersetup{%
     pdfborder = {1 1 0}
}

\newcommand{\FF}{\mathbb F}

\newcommand{\PP}{\mathbb P}
\newcommand{\A}{\mathbb A}
\newcommand{\Q}{\mathbb Q}
\newcommand{\sL}{\mathcal L}

\renewcommand{\O}{\mathcal O}

\newcommand\Pic{\mathop{\rm Pic} \nolimits}

\newcommand\Gal{\mathop{\rm Gal} \nolimits}

\newcommand\Sym{\mathop{\rm Sym} \nolimits}

\newcommand\kbar{\overline{k}}

\newtheorem{theorem}{Theorem}[section]

\newtheorem{lemma}[theorem]{Lemma}

\theoremstyle{remark}
\newtheorem{remark}[theorem]{Remark}

\begin{document}

\begin{abstract}
We prove that every del Pezzo surface of degree two over a finite field is unirational, building on the work of Manin and an extension by Salgado, Testa, and V\'arilly-Alvarado, who had proved this for all but three surfaces. Over general fields of characteristic not equal to two, we state sufficient conditions for a del Pezzo surface of degree two to be unirational. 
\end{abstract}

\title[Unirationality of del Pezzo surfaces of degree two]{Unirationality of del Pezzo surfaces of degree two over finite fields}
\author{Dino Festi, Ronald van Luijk}
\date{}
\maketitle
\section{Introduction}
A \textit{del Pezzo surface} is a smooth, projective, geometrically integral variety $X$ of which the anticanonical divisor $-K_X$ is ample. 
We define the \textit{degree} of a del Pezzo surface $X$ as the self intersection number of $K_X$, that is, $\deg X = K_X^2$. If $k$ is an algebraically closed field, then every del Pezzo surface of degree $d$ over $k$ is isomorphic to $\PP^1 \times \PP^1$ (with $d=8$), or to $\PP^2$ blown up in $9-d$ points in general position. 

Over arbitrary fields, the situation is more complicated and del Pezzo surfaces need not be birationally equivalent with $\PP^2$. We therefore look at the weaker notion of unirationality.
We say that a variety $X$ of dimension $n$ over a field $k$ is \textit{unirational} if there exists a dominant rational map $\PP^n \dashrightarrow X$, defined over $k$.
We prove the following theorem. 

\begin{theorem}\label{Thm:main}
Every del Pezzo surface of degree 2 over a finite field is unirational.
\end{theorem}

The analog for higher degree holds over any field. 
Works of B.~Segre, Yu.~Manin, J.~Koll\'ar, M.~Pieropan, and A.~Knecht prove that every del Pezzo surface of degree $d\geq 3$, defined over any field $k$, is unirational, provided that the set $X(k)$ of rational points is non-empty. For references,  see \cite{Seg43, Seg51} for $k=\Q$ and $d=3$, see \cite[Theorem 29.4 and 30.1]{Man86} for $d\geq 3$ with the extra assumption for $d\in \{3,4\}$ that~$k$ has enough elements. 
See \cite[Theorem 1.1]{Kol02} for $d=3$ and a general ground field. The earliest reference we could find for $d=4$ and a general ground field is \cite[Proposition 5.19]{Pie12}. Independently, for $d=4$, \cite[Theorem 2.1]{Kne13} covers all finite fields. 
Since all del Pezzo surfaces over finite fields have a rational point (see \cite[Corollary 27.1.1]{Man86}), this implies that  every del Pezzo surface of degree at least $3$ over a finite field is unirational. 

Most of the work to prove Theorem~\ref{Thm:main} was already done. 
Building on work by Manin (see \cite[Theorem 29.4]{Man86}), C.~Salgado, D.~Testa, and A.~V\'arilly-Alvarado prove that all 
del Pezzo surfaces of degree $2$ over a finite field are unirational, except possibly for three isomorphism classes 
of surfaces (see \cite[Theorem 1]{STVA13}). 
In Section~\ref{S:pfmain}, we will present the three difficult surfaces and 
show that these are also unirational, thus proving Theorem~\ref{Thm:main}. 

Before that, in Section~\ref{S:basic}, we will recall the basics about del Pezzo surfaces of degree $2$, including the fact that 
the linear system associated to the anti-canonical divisor induces a finite morphism to $\PP^2$ of degree $2$. 
We call this morphism the {\em anti-canonical morphism} associated to $X$. This allows us to state the second main theorem. 

\begin{theorem}\label{C:NonRam}
Suppose $k$ is a field of characteristic not equal to $2$. 
Let $X$ be a del Pezzo surface of degree~$2$ over~$k$, and let $\pi\colon X \to \PP^2$ be 
its anti-canonical morphism. Assume that $X$ has a $k$-rational point, say $P$. 
Let $C\subset \PP^2$ be a geometrically integral curve over $k$ of degree $d \geq 2$ and 
suppose that $\pi(P)$ is a point of multiplicity $d-1$ on $C$. Suppose, moreover, 
that $C$ intersects the branch locus $B$ of the morphism $\pi$ with even multiplicity everywhere. Then the following statements hold.
\begin{enumerate}
\item If $\pi(P)$ is not contained in $B$, then $X$ is unirational.
\item If $\pi(P)$ is contained in $B$, and it is an ordinary singular point on $C$ and we have  $d\in \{3,4\}$, then there 
exists a field extension $\ell$ of $k$ of degree at most $2$ for which the preimage $\pi^{-1}(C_\ell)$ 
is birationally equivalent with $\PP^1_\ell$;  for each such field $\ell$, the surface $X_\ell$ is unirational.
\end{enumerate}
\end{theorem}

The main tool for the proof of both theorems is Lemma~\ref{P1Lemma1} (that is,~\cite[Theorem~17]{STVA13}), which states that, outside characteristic $2$, a del Pezzo surface of degree $2$
is unirational if it contains a rational curve.  We prove Theorem~\ref{C:NonRam} in Section~\ref{S:pfmaintwo} by showing that, under the hypotheses of Theorem~\ref{C:NonRam}, 
the pull-back of the curve $C$ to $X$ contains a rational component. 
Manin's original construction, and the generalisation by Salgado, Testa, and V\'arilly-Alvarado, 
produces a rational curve that corresponds to case (1) of Theorem~\ref{C:NonRam}, with $4-d$ equal to the number of exceptional curves that $P$ lies on. 

For the three difficult surfaces one can use case (2) of Theorem~\ref{C:NonRam} (see Remark~\ref{ex:quarticcurves}). 
Here we benefit from the fact that if $k$ is a finite field,  then any curve that becomes birationally equivalent with $\PP^1$ over an extension of $k$, already is 
birationally equivalent with $\PP^1$ over $k$ itself.

For interesting examples and more details about 
the proof of Theorem~\ref{C:NonRam}, Manin's construction, as wel as a generalisation of Theorem~\ref{C:NonRam}, 
we refer the reader to an extended version of this paper \cite{FvL14}.

The authors would like to thank Bjorn Poonen, Damiano Testa and Anthony V\'arilly-Alvarado for useful conversations.

\section{Del Pezzo surfaces of degree two}\label{S:basic}

The statements in this section are well known and we will use them freely.  
Let $X$ be a del Pezzo surface of degree $2$ over a field $k$ with canonical divisor~$K_X$.
The Riemann-Roch spaces $\sL(-K_X)$ and $\sL(-2K_X)$ have dimension $3$ and $7$, respectively. 
Let $x,y,z$ be generators of $\sL(-K_X)$ and choose an element $w \in \sL(-2K_X)$ that is not contained 
in the image of the natural map $\Sym^2 \sL(-K_X) \to \sL(-2K_X)$. 
Then $X$ embeds into the weighted projective space $\PP= \PP(1,1,1,2)$ with coordinates $x,y,z$, and $w$.
We will identify $X$ with its image in $\PP$, which is a smooth surface of degree~$4$. 
Conversely, every smooth surface of degree~$4$ in $\PP$ is a del Pezzo surface of 
degree~$2$. There are homogeneous polynomials
$f,g \in k[x,y,z]$ of degrees $2$ and $4$, respectively, such that $X \subset \PP$ is given by 
\begin{equation}\label{eq:geneq}
w^2+fw=g.
\end{equation}
If the characteristic of $k$ is not $2$, then after completing the square on the left-hand side, 
we may assume $f=0$. For more details and proofs of these facts, see 
\cite[Section~III.3, Theorem~III.3.5]{Kol96} and \cite[Section IV.24]{Man86}.

The restriction to $X$ of the $2$-uple embedding 
$\PP\to \PP^6$ corresponds to the complete linear system $|-2K_X|$.
Every hyperplane section of $X \subset \PP$ is linearly equivalent with~$-K_X$. 
The projection $\PP \dashrightarrow \PP^2$ onto the first three coordinates
restricts to a finite, separable morphism $\pi_X\colon X \to \PP^2$ of degree $2$, which 
corresponds to the complete linear system $|-K_X|$. This is the anti-canonical morphism
mentioned in the introduction.

The morphism $\pi_X$ is ramified above the branch locus $B_X \subset \PP^2$ given by $f^2+4g=0$. 
If the characteristic of $k$ is not $2$, then $B_X$ is a smooth curve. 
We denote the ramification locus $\pi^{-1}(B_X)$ of $\pi_X$ by~$R_X$. 
As for every double cover, the morphism $\pi_X$ induces
an involution $\iota_X\colon X \to X$ that sends a point $P\in X$ to the unique second point in 
the fiber $\pi_X^{-1}(\pi_X(P))$, or to $P$ itself if $\pi_X$ is ramified at~$P$. 
If $X$ is clear from the context, then we sometimes leave out the subscripts and write 
$\pi, \iota$, $B$, and $R$ for $\pi_X, \iota_X, B_X$, and $R_X$, respectively.

\section{Proof of the first main theorem}\label{S:pfmain}

Set $k_1=k_2=\FF_3$ and $k_3 = \FF_9$. Let $\gamma \in k_3$ denote an element satisfying $\gamma^2=\gamma+1$.
Note that $\gamma$ is not a square in $k_3$. 
For $i\in \{1,2,3\}$, we define the surface $X_i$ in $\PP=\PP(1,1,1,2)$ with coordinates $x,y,z,w$ over~$k_i$ by  
\begin{align*}
X_1 \, \colon \; -w^2 \; &= \; (x^2+y^2)^2+y^3z-yz^3,\\
X_2 \, \colon \; -w^2 \; &= \; x^4+y^3z-yz^3,\\
X_3 \, \colon \; \gamma w^2 \; &= \; x^4+y^4+z^4.
\end{align*}
These surfaces are smooth, so they are del Pezzo surfaces of degree $2$.  
C.\ Salgado, D.\ Testa, and A.\ V\'arilly-Alvarado proved the following result.

\begin{theorem}\label{TSVAfinite}
Let $X$ be a del Pezzo surface of degree 2 over a finite field. If $X$ is not isomorphic to 
$X_1, X_2$, and $X_3$, then $X$ is unirational. 
\end{theorem}
\begin{proof}
See \cite[Theorem 1]{STVA13}. 
\end{proof}

We will use the following lemma to prove the complementary statement, 
namely that $X_1, X_2$, and $X_3$ are unirational as well. 

\begin{lemma}\label{P1Lemma1}
Let $X$ be a del Pezzo surface of degree 2 over a field~$k$. 
Suppose that $\rho \colon \PP^1 \to X$ is a nonconstant morphism; if the characteristic of $k$ is $2$ and the 
image of $\rho$ is contained in the ramification divisor $R_X$, then 
assume also that the field $k$ is perfect. Then $X$ is unirational.
\end{lemma}
\begin{proof}
See \cite[Theorem 17]{STVA13}.
\end{proof}

For $i \in \{1,2,3\}$, we define a morphism $\rho_i \colon \PP^1 \to X_i$ by 
extending the map $\A^1(t) \to X_i$ given by  
$$
t \mapsto (x_i(t):y_i(t):z_i(t):w_i(t)),
$$ 
where 
\begin{center}
\begin{minipage}{2in}
\begin{align*}
x_1(t)&=t^2 (t^2-1) ,\\
y_1(t)&=t^2 (t^2-1)^2,\\
z_1(t)&=t^8 -t^2 + 1,\\
w_1(t)&=t(t^2-1)(t^4+1)(t^8+1),
\end{align*}
\end{minipage}
\begin{minipage}{1.8in}
\begin{align*}
x_2(t)&=t(t^2+1)(t^4-1),\\
y_2(t)&=-t^4,\\
z_2(t)&=t^8+1,\\
w_2(t)&=t^2(t^2 + 1)(t^{10}-1),
\end{align*}
\end{minipage}
\begin{minipage}{1.8in}
\begin{align*}
x_3(t)&=(t^4+1)(t^2-\gamma^3),\\
y_3(t)&=(t^4-1)(t^2+\gamma^3),\\
z_3(t)&=(t^4+\gamma^2)(t^2-\gamma),\\
w_3(t)&=\gamma^2t(t^8-1)(t^2+\gamma).
\end{align*}
\end{minipage}
\end{center}
\smallskip

It is easy to check for each $i$ that the morphism $\rho_i$ is well defined, that is, the polynomials $x_i,y_i,z_i$, and $w_i$ satisfy the equation of $X_i$, and that $\rho_i$ is non-constant. 

\begin{theorem}\label{Thm:X123uni}
The del Pezzo surfaces $X_1, X_2$, and $X_3$ are unirational.
\end{theorem}
\begin{proof}
By Lemma~\ref{P1Lemma1}, the existence of $\rho_1, \rho_2$, and $\rho_3$ implies that $X_1,X_2$, and $X_3$ are unirational.
\end{proof}

\begin{proof}[Proof of Theorem \ref{Thm:main}]
This follows from Theorems \ref{TSVAfinite} and~\ref{Thm:X123uni}.
\end{proof}

\section{Proof of the second main theorem}\label{S:pfmaintwo}

If $C$ is a plane curve with an ordinary singularity $Q$ and $\tilde{C}$ is the normalisation of $C$, 
then we can think of the points of $\tilde{C}$ above $Q$ as corresponding with the branches of $C$ 
through $Q$. The intersection multiplicity of $C$ with another plane curve $B$ at $Q$ is then the sum 
of the intersection multiplicities of $B$ with all the branches of $C$ through $Q$. This point of view 
is used in the following proof. For more technical details about this approach, see the extended version 
of this paper \cite{FvL14}.

\begin{proof}[Proof of Theorem \ref{C:NonRam}]
Let $\iota\colon X \to X$ denote the involution associated to the double cover $\pi$.
Set $Q = \pi(P)$. 
Projection away from the point $Q\in C \subset \PP^2$ yields a birational map $C \dashrightarrow \PP^1$ whose
inverse $\vartheta \colon \PP^1 \to C$ can be identified with the normalisation map of $C$. The
map $\vartheta$ restricts to an isomorphism $\PP^1 \setminus \vartheta^{-1}(Q) \to C \setminus \{Q\}$, 
and $C$ is smooth away from $Q$. 
Let $D = \pi^{-1}(C)$ be the inverse image of $C$ under~$\pi$, and let $\tilde{D}$ be its normalisation. 
Then $\pi$ induces a double cover $\tilde{\pi} \colon \tilde{D} \to \PP^1$. 

Let $S \in \PP^1$ be a point and set $T = \vartheta(S)\in C$. 
The curve $B$ is given locally around $T$ by the vanishing of a rational function on $\PP^2$
that is regular at $T$. We let $h$ denote the image of such a function in the local ring $\O_{C,T}$ of 
$T$ in $C$. 

If $T \neq Q$, then $T$ is a smooth point of $C$, so the ring $\O_{C,T}=\O_{\PP^1,S}$ is a discrete valuation ring. 
In this case, the valuation of $h$ equals the intersection multiplicity of $B$ and $C$ at $T$, which is even.  
Since the characteristic of $k$ is not $2$, this implies that adjoining a square root 
of $h$ to $\O_{C,T}$ yields an unramified extension, so the morphism $\tilde{\pi}\colon \tilde{D} \to \PP^1$ 
is not ramified above $S$ when $T\neq Q$. 

Suppose that $Q$ is not contained in $B$. Then for $T=Q$, the element $h$ is a unit in the local ring $\O_{C,T}$,
and therefore also in the ring extension $\O_{\PP^1,S}$. Hence, as before, since the characteristic of $k$ is not $2$, 
this implies that the morphism $\tilde{\pi}$ is not ramified above $S$. This means that $\tilde{\pi}$ is unramified. 
Since $\PP^1_{\kbar}$ has no nontrivial unramified covers, this means that the curve $\tilde{D}$, and hence the curve 
$D \subset X$, splits into two 
components over some quadratic extension $\ell$ of $k$. Exactly one of the components of $D$ contains the rational 
point $P$ and the other component contains $\iota(P)$. This implies that the Galois group $\Gal(\ell/k)$ 
sends each component to itself, 
so these components are defined over $k$. Each maps isomorphically to~$C$, so $X$ contains a curve that is 
birationally equivalent to $\PP^1$ and therefore $X$ is unirational by Lemma~\ref{P1Lemma1}. This proves (1). 

Suppose that $Q$ is contained in $B$ and that it is an ordinary singular point on $C$. Then 
$\vartheta^{-1}(Q)$ consists of exactly $d-1$ points over $\kbar$, each corresponding to the
tangent direction of one of the $d-1$ branches of $C$ at $Q$. At most one of tangent directions 
is tangent to $B$, so at least $d-2$ of the branches intersect $B$ with multiplicity $1$. The total intersection 
multiplicity of $B$ and $C$ at~$Q$ is even. If $d$ is odd, then the contribution $(d-2)\cdot 1$ 
of the $d-2$ branches with intersection multiplicity $1$ is odd, so the last branch intersects $B$ with 
odd multiplicity as well; hence all $d-1$ branches intersect $B$ with odd multiplicity, which implies 
that $\tilde{\pi} \colon \tilde{D} \to \PP^1$ is ramified above all $d-1$ points above $Q$. If $d$ is even, 
then the contribution of the $d-2$ branches of $C$ that intersect~$B$ with multiplicity $1$ is even as well, 
so the last branch intersects $B$ with even multiplicity; as before, this means that $\tilde{\pi}$ is not 
ramified above the point in $\vartheta^{-1}(Q) \subset \PP^1$ that corresponds to this last branch, 
so $\tilde{\pi}$ is ramified above exactly $d-2$ of the $d-1$ points above $Q$. For $d \in \{3,4\}$, these two 
cases ($d$ odd or even) imply that 
the map $\tilde{\pi} \colon \tilde{D} \to \PP^1$ is ramified at exactly two points, so $\tilde{D}$ is a geometrically 
integral curve of genus $0$ by the theorem of Riemann-Hurwitz. Indeed, this implies that there
is a field extension $\ell$ of $k$ of degree at most $2$ for which $\tilde{D}_\ell$, and thus 
$D_\ell = \pi^{-1}(C_\ell)$, is birationally equivalent with~$\PP^1_\ell$. For each such field, 
the surface $X_\ell$ is unirational by Lemma~\ref{P1Lemma1}. This proves (2).  
\end{proof}

\begin{remark}\label{ex:quarticcurves}
Let the surfaces $X_1,X_2,X_3$ and the morphisms $\rho_1,\rho_2,\rho_3$ be as in the previous section. 
Take any $i \in \{1,2,3\}$.  
Set $A_i=\rho_i(\PP^1)$ and $C_i = \pi_i(A_i)$, where $\pi_i = \pi_{X_i} \colon X_i \to \PP^2$ is as 
described in the previous section.
By Remark 2 of \cite{STVA13}, the surface $X_i$ is minimal, and the Picard group 
$\Pic X_i$ is generated by the class of the anticanonical divisor $-K_{X_i}$.
The same remark states that the linear system $|-nK_{X_i}|$ does not contain a geometrically integral 
curve of geometric genus zero for $n\leq 3$ if $i\in \{1,2\}$, nor for $n\leq 2$ if $i=3$. 
For $i \in \{1,2\}$, the curve $A_i$ has degree~$8$, so it is contained in the linear system $|-4K_{X_i}|$.
The curve $A_3$ has degree~$6$, so it is contained in the linear system $|-3K_{X_i}|$. 
This means that the curve $A_i$ has minimal degree among all rational curves on $X_i$.
The restriction of $\pi_i$ to $A_i$ is a double cover $A_i \to C_i$. 
The curve $C_i \subset \PP^2$ has degree~$4$ for $i \in \{1,2\}$ and degree $3$ for $i=3$, 
and $C_i$ is given by the vanishing of $h_i$, with
\begin{align*}
h_1=&x^4 + x y^3 + y^4 - x^2 y z - x y^2 z,\\
h_2=&x^4 - x^2 y^2 - y^4 + x^2 y z + y z^3, \\
h_3=&x^2y + xy^2 + x^2z - xyz + y^2z - xz^2 - yz^2 - z^3.
\end{align*} 
For $i \in \{1,2\}$, the curve $C_i$ has an ordinary triple point $Q_i$, with $Q_1 = (0:0:1)$, $Q_2 = (0:1:1)$. 
The curve $C_3$ has an ordinary double point at $Q_3=(1:1:1)$. 
For all $i$, the point $Q_i$ lies on the branch locus $B_i = B_{X_i}$.   

Using the polynomial $h_i$, one can check that the curve $C_i$ intersects the branch locus $B_i$ with even multiplicity everywhere. 
In fact, had we {\em defined}
$C_i$ by the vanishing of $h_i$, then one would easily check that $C_i$ satisfies the conditions of 
part (2) of Theorem~\ref{C:NonRam}.
This gives an alternative proof of unirationality of $X_i$ without the
need of the explicit morphism $\rho_i$; here we may
use the fact that if $k$ is a finite field, 
then any curve that becomes birationally equivalent to $\PP^1$ over an extension of~$k$, already is birationally equivalent with $\PP^1$ over $k$.
Indeed, in practice we first found 
the curves $C_1$, $C_2$, and $C_3$, and then constructed the parametrisations $\rho_1,\rho_2,\rho_3$, 
which allow for the more direct proof that we gave of Theorem~\ref{Thm:X123uni} in the previous section.
\end{remark}

\bibliographystyle{alpha}
\bibliography{dP2sBiblio}{}

\end{document}